\documentclass[11pt]{article}
\usepackage{graphicx}
\usepackage[margin=0.5in]{geometry}
\usepackage{bm}
\usepackage{amsthm}
\usepackage{amscd}
\usepackage{amsbsy}
\usepackage{amsmath}
\usepackage{amsfonts}
\usepackage{amssymb}
\usepackage{calligra}
\usepackage{xcolor}
\usepackage{float}
\usepackage[T1]{fontenc}
\usepackage{multirow}
\usepackage{mathrsfs}
\usepackage{mathtools}
\usepackage{epstopdf}
\usepackage{psfrag}
\usepackage{subfigure}
\usepackage{booktabs}
\usepackage{listings}
\usepackage{longtable}
\usepackage{latexsym}
\usepackage{arydshln}
\usepackage{enumerate}
\usepackage{epsfig}
\usepackage{cite}
\usepackage{verbatim}
\usepackage{overpic}
\usepackage{abstract}
\allowdisplaybreaks[4]

\date{}
\title{The necessary and sufficient conditions for the real Jacobian conjecture}
\author{Yuzhou Tian and Yulin Zhao\footnote{Corresponding author. E-mail address: tianyzh3@mail2.sysu.edu.cn (Y. Tian), mcszyl@mail.sysu.edu.cn (Y. Zhao).}\\
\it\footnotesize School of Mathematics (Zhuhai),\ Sun Yat-sen University,\ Zhuhai\ 519082, P.R.\ China}

\newtheorem {theorem*}{Theorem}
\newtheorem {theorem} {Theorem}
\newtheorem{example}{Example}

\newtheorem{proposition}{Proposition}
\newtheorem{lemma}{Lemma}

\newtheorem{remark}{Remark}

\numberwithin{equation}{section}

\usepackage[colorlinks,
            CJKbookmarks=true,
            linkcolor=blue,
            anchorcolor=red,
             urlcolor=blue,
            citecolor=blue
            ]{hyperref}

\begin{document}
\maketitle
\noindent {\bf Abstract} We focus on investigating the real Jacobian conjecture. This conjecture claims that if $F=\left(f^1,\ldots,f^n\right):\mathbb{R}^n\rightarrow \mathbb{R}^n$ is a polynomial map such that $\det DF$ is nowhere zero, then $F$ is a global injective.

This paper consists of two parts. The first part is to study  the two-dimensional real Jacobian conjecture via the method of the qualitative theory of dynamical systems. We provide some necessary and sufficient conditions such that the two-dimensional real Jacobian conjecture holds. By Bendixson compactification, an induced polynomial differential system can be obtained from the Hamiltonian system associated to polynomial map $F$. We prove that the following statements are equivalent: (A) $F$ is a global injective; (B) the origin of induced system is a center; (C) the origin of induced system is a monodromic singular point; (D) the origin of induced system has no hyperbolic sectors; (E) induced system has a $C^k$ first integral with an isolated minimun at the origin and $k\in\mathbb{N}^{+}\cup\{\infty\}$. The above conditions (B)-(D) are local dynamical conditions. Moreover, applying the above results we present a necessary and sufficient condition for the validity of the two-dimensional real Jacobian conjecture, which is an algebraic criterion. By definition a criterion function, $F$ is a global injective if and only if the limit of criterion function is infinite as $\left|x\right|+\left|y\right|$ tends to infinity. This algebraic criterion improves the main result of Braun et al [J. Differential Equations {\bf 260} (2016) 5250-5258]. 

In the second part, the necessary and sufficient conditions on the $n$-dimensional real Jacobian conjecture is obtained. Using the tool from the nonlinear functional analysis, $F$ is a global injective if and only if $\parallel F\left(\mathbf{x}\right)\parallel$ approaches to infinite as $\parallel\mathbf{x}\parallel\rightarrow\infty$, which is a generalization of  the above  algebraic criterion.  As an application, we give an alternate proof of the Cima's result on the $n$-dimensional real Jacobian conjecture [Nonlinear Anal. {\bf 26} (1996) 877-885]. 
\smallskip

\noindent {\bf 2020 Math Subject Classification: } Primary 14R15.  Secondary 08B30.  Tertiary 34C05

\smallskip


\section{Introduction and main results}\label{se-1}
Let $F\left(\mathbf{x}\right)=\left(f^1\left(\mathbf{x}\right),\ldots,f^n\left(\mathbf{x}\right)\right):\mathbb{R}^n\rightarrow\mathbb{R}^n$ be a smooth map with the Jacobian determinant $\det DF\left(\mathbf{x}\right)\neq0$  for all $\mathbf{x}=\left(x_1,\ldots,x_n\right)\in\mathbb{R}^n$. Obviously, the map $F$ is a local diffeomorphism. However, it is not always global injective in $\mathbb{R}^n$. Actually, one can impose suitable conditions to guarantee that $F$ is a global diffeomorphism,  see for example \cite{cobo2002injectivity,fernandes2004global,Braun2017On,plastock1974homeomorphisms} and references therein.

In algebraic geometry, the well-known \emph{Jacobian conjecture} is to state that if $F:\mathbb{C}^n\rightarrow\mathbb{C}^n$ is a polynomial map with $\det DF$ \emph{a non-zero constant}, then $F$ is a global injective. This conjecture was first introduced by Keller in 1939, and up to now it is still open problem. Smale \cite{MR1631413} in 1998 listed Jacobian conjecture as the 16th of 18 great mathematical problems for the 21th century. For Jacobian conjecture there are many positive partial results, see \cite{MR714106,MR3866897,MR1487631,MR2247890,MR3556520,MR3278901,MR714105,MR3652577}, etc. The investigation of Jacobian conjecture leads to a stream of valuable results concerning polynomial  automorphisms, as shown in survey \cite{bass1982jacobian} and book \cite{van2012polynomial}, etc.

From now on, we consider the polynomial map $F$ defined in $\mathbb{R}^n$. Another famous conjecture, the \emph{real Jacobian conjecture} claims that if $F:\mathbb{R}^n\rightarrow\mathbb{R}^n$ is a polynomial map with \emph{nonvanishing Jacobian determinant}, then $F$ is a global injective, see \cite{MR713265}. Unfortunately, this conjecture is false. In 1994, Pinchuk \cite{MR1292168} provided a counterexample which is a non-injective polynomial map $F$ in $\mathbb{R}^2$ with nonvanishing Jacobian determinant.  Nevertheless, the real Jacobian conjecture has still attracted the interest of numerous mathematicians, especially exploring conditions such that this conjecture holds. Based on the structure of polynomial map $F$, the authors in \cite{cima1995global,MR1362759} give sufficient conditions. Gwo\'{z}dziewicz in \cite{MR1839866} obtained that the two-dimensional real Jacobian conjecture (i.e., the polynomial map $F=\left(f,g\right)$) holds if the degrees of $f$ and $g$ are less than or equal to $3$. Braun et al. \cite{MR2552779,MR3514314} generalized this result by showing that the conjecture is true if the degree of $f$ is at most $4$, independently of the degree of $g$. In the above mentioned papers, the main technique relates algebra, analysis and geometry.

In the first part of this paper, we study the two-dimensional real Jacobian conjecture. To describe our main results, we first introduce some notation. Consider a two-dimensional autonomous differential system
\begin{align}\label{eq-3}
&\dot{x}=P\left(x,y\right),\quad\dot{y}=Q\left(x,y\right).
\end{align}
Denoted by $\mathscr{X}=\left(P,Q\right)$ the vector field associated to system \eqref{eq-3}. The vector field $\mathscr{X}$ is $C^k$ with $k\in \mathbb{N}^+\cup\{\infty\}$ if $P(x,y)$ and $Q(x,y)$ are $C^k$. Let $U$ be an open set of $\mathbb{R}^2$. A non-locally constant function $H: U\rightarrow \mathbb{R}$ is called a \emph{first integral} of $\mathscr{X}$ if it is constant along any solution curve of $\mathscr{X}$ contained in $U$. We denote by
$b\left(\mathscr{X}\right)$ the \emph{Bendixson compactification} (see subsection \ref{sub-1}) of vector field $\mathscr{X}$. In particular, if $P\left(x,y\right),Q\left(x,y\right)$ are real polynomials in variables $x$ and $y$ with $d=\max\{\text{deg}P,\text{deg}Q\}$, the expression of $b\left(\mathscr{X}\right)$ is given by
\begin{align*}
\begin{cases}
\dot{u}=\left(u^2+v^2\right)^d\left[\left(v^2-u^2\right)P\left(\dfrac{u}{u^2+v^2},\dfrac{v}{u^2+v^2}\right)-2uvQ\left(\dfrac{u}{u^2+v^2},\dfrac{v}{u^2+v^2}\right)\right],\\
\specialrule{0em}{3pt}{3pt}
\dot{v}=\left(u^2+v^2\right)^d\left[\left(u^2-v^2\right)Q\left(\dfrac{u}{u^2+v^2},\dfrac{v}{u^2+v^2}\right)-2uvP\left(\dfrac{u}{u^2+v^2},\dfrac{v}{u^2+v^2}\right)\right],
\end{cases}
\end{align*}
see subsection \ref{sub-1} for more details.

Let $q\in\mathbb{R}^2$ be a singular point of an analytic vector field $\mathscr{X}$ in $\mathbb{R}^2$. The singular point $q$ is \emph{monodromic} if there exists a neighborhood of $q$ such that the orbits of the vector field turn around $q$ either in forward or in backward time. We say that the singular point $q$ is a \emph{center} if there is a neighborhood of $q$ which is filled up with periodic orbits. The \emph{period annulus} of the center $q$ is the maximal neighbourhood $U$ of $q$ such that all the orbits contained in $U\setminus\{q\}$ are periodic. The center $q$ is a \emph{global center} if its period annulus is the whole $\mathbb{R}^2$. A singular point $q$ is called a \emph{focus} if all orbits in a neighborhood of $q$ spirally approach this singular point either in forward or in backward time.

Sabatini \cite{MR1636592}  gave the following dynamical result.
\begin{theorem}\label{th-1}
  Let $F=\left(f\left(x,y\right),g\left(x,y\right)\right):\mathbb{R}^2\rightarrow \mathbb{R}^2$ be a polynomial map with nowhere zero Jacobian determinant such that $F\left(0,0\right)=\left(0,0\right)$. Then the following statements are equivalent.
\begin{itemize}
\item [(a)] The origin is a global center for the Hamiltonian polynomial vector field
    \begin{align}\label{eq-1}
&\mathcal{X}\triangleq\left(-ff_y-gg_y,ff_x+gg_x\right).
    \end{align}
    \item [(b)] $F$ is a global diffeomorphism of the plane onto itself.
 \end{itemize}
\end{theorem}
This theorem provides a \emph{global dynamical condition} such that $F$ is a global injective. Recently, Braun and Llibre proved in \cite{MR3443401} that if the homogeneous terms of higher degree of $f$ and $g$ do not have real linear factors in common and $\text{deg}f=\text{deg}g$, then $F$ is a global injective. Later on, Braun et al. \cite{MR3448779} improved this result in the following theorem.
\begin{theorem}\label{th-2}
  Let $F=\left(f,g\right):\mathbb{R}^2\rightarrow \mathbb{R}^2$ be a polynomial map such that \emph{det} $DF\left(x,y\right)$ is nowhere zero and $F\left(0,0\right)=\left(0,0\right)$. If the higher homogeneous terms of the polynomials $ff_x+gg_x$ and $ff_y+gg_y$ do not have real linear factors in common, then $F$ is a global injective.
\end{theorem}
Itikawa et al. \cite{Jackson} give two new classes of polynomial maps satisfying the real Jacobian conjecture in $\mathbb{R}^2$. A new proof of Pinchuk map which is a non-injective can be found in \cite{MR3866202}. In these works, their proofs rely only on the qualitative theory of planar differential systems, following ideas inspired by Theorem \ref{th-1}. We note that the essential of their proofs are  to characterize the global dynamical behavior of Hamiltonian polynomial vector field $\mathcal{X}$. As we know, the global dynamical analysis of Hamiltonian polynomial vector field in many cases are hard and tedious, because we need to get the local  behavior at the all finite and infinite singular points, and to determine their separatrix configurations.  For this reason, it is natural to ask whether there exist \emph{local dynamical conditions} to ensure that $F$ is a global injective in $\mathbb{R}^2$.  

Our first result of this paper provides necessary and sufficient conditions for the validity of the two-dimensional real Jacobian conjecture.
\begin{theorem}\label{th-3}
  Let $F=\left(f,g\right):\mathbb{R}^2\rightarrow \mathbb{R}^2$ be a polynomial map with nowhere zero Jacobian determinant such that $F\left(0,0\right)=\left(0,0\right)$. Denote by $b\left(\mathcal{X}\right)$ the Bendixson compactification of Hamiltonian vector field $\mathcal{X}$ defined in \eqref{eq-1}. Then the following statements are equivalent.
  \begin{enumerate}[(a)]
    \item $F$ is a global diffeomorphism of the plane onto itself.
    \item The origin of the polynomial vector field $b\left(\mathcal{X}\right)$ is a center.
    \item The origin of the polynomial vector field $b\left(\mathcal{X}\right)$ is a monodromic singular point.
    \item The origin of the polynomial vector field $b\left(\mathcal{X}\right)$ has no hyperbolic sectors.
    \item The polynomial vector field $b\left(\mathcal{X}\right)$ has a $C^k$ first integral with an isolated minimun at the origin, where $k\in\mathbb{N}^{+}\cup\{\infty\}.$
  \end{enumerate}
\end{theorem}
\begin{remark}\label{re-2}
  \emph{For the condition $(a)$ of Theorem \ref{th-1}, the origin of $\mathcal{X}$ must be a global center. For the condition $(b)$ of Theorem \ref{th-3}, the origin of $b\left(\mathcal{X}\right)$ is a local center, not necessary global center. The above conditions $(c)$ and $(d)$ are also local. The last condition $(e)$ is from the point of view of integrability. The vector field $b\left(\mathcal{X}\right)$ always exists a first integral (see Section \ref{se-6}), but such a first integral in general cannot be extended to the origin of $b\left(\mathcal{X}\right)$. Theorem \ref{th-3} implies that the local dynamical behavior of $b\left(\mathcal{X}\right)$ at the origin determines fully whether polynomial map $F$ is a global injective. So it allows us to investigate real Jacobian conjecture by some elementary dynamical tools which is the local analysis of singular points, for example blow up technique. In fact, the origin of vector field $b\left(\mathcal{X}\right)$ is \emph{degenerate} (see proof of Theorem \ref{th-3}), that is, its linear part identically zero. The two classical problems for degenerate singular point are respectively \emph{monodromy problem} to decide whether it is of focus-center type, and \emph{stability problem} to distinguish  between a center and a focus. For degenerate singular point, these two problems are very complicated  in general vector field, see \cite{MR1758996,MR1912074,MR2558171}. It is worth to notice that the monodromy and stability problems of vector field $b\left(\mathcal{X}\right)$ at the origin are completely solved under the assumption of Theorem \ref{th-3}.}
\end{remark}

Using Theorem \ref{th-3}, we present a necessary and sufficient condition such that the two-dimensional real Jacobian conjecture holds, which is an \emph{algebraic criterion}. Our second result is described as follows.
\begin{theorem}\label{th-14}
 Let $F=\left(f,g\right):\mathbb{R}^2\rightarrow \mathbb{R}^2$ be a polynomial map with nowhere zero Jacobian determinant such that $F\left(0,0\right)=\left(0,0\right)$. Define a polynomial criterion function as follows:
\begin{align}\label{eq-32}
&I\left(x,y\right)=f^2\left(x,y\right)+g^2\left(x,y\right).
\end{align}
Then $F$ is a global injective if and only if
\begin{align}\label{eq-33}
&\lim_{\left|x\right|+\left|y\right|\rightarrow+\infty}I(x,y)=+\infty.
\end{align}
\end{theorem}
The following example \cite{MR1362759} shows that Theorem \ref{th-14} improves Theorem \ref{th-2}.
\begin{example}\label{ex-1}
  \emph{Consider the  polynomial map $F=(f,g):\mathbb{R}^2\rightarrow\mathbb{R}^2$ with $f=y+y^3$ and $g=x+xy^2$. Here $\det DF=-\left(1+y^2\right)\left(1+3y^2\right)$. The higher homogeneous terms of
\begin{align*}
&ff_x+gg_x=x+2xy^2+xy^4,\quad ff_y+gg_y=y+2y\left(x^2+2y^2\right)+y^3\left(2x^2+3y^2\right)
\end{align*}
have $y$ as a common factor. This map $F$ does not satisfy the condition of Theorem \ref{th-2}.  The criterion function is given by
$$I(x,y)=\left(x^2+y^2\right)\left(1+y^2\right)^2\geq x^2+y^2.$$
Obviously, $\lim_{\left|x\right|+\left|y\right|\rightarrow+\infty}I(x,y)=+\infty$. So $F$ is a global injective.}
\end{example}

After completing the proof of Theorem \ref{th-14}, a natural idea appears for us, i.e. how to generalize it from $\mathbb{R}^2$ to $\mathbb{R}^n$. Along the method to prove Theorem \ref{th-14}, we attempt to generalize it. There are essential differences between the dynamic properties of two-dimensional  and high-dimensional. Our accidental meeting the tool from nonlinear functional analysis led to the following result.
\begin{theorem}\label{th-16}
 Let $F=(f^1,\ldots,f^n):\mathbb{R}^n\rightarrow \mathbb{R}^n$ be a polynomial map with nowhere zero Jacobian determinant such that $F\left(0,\ldots,0\right)=\left(0,\ldots,0\right)$.
Then $F$ is a global injective if and only if
\begin{align}\label{eq-47}
&\lim_{\parallel \mathbf{x}\parallel\rightarrow\infty}\parallel F\left(\mathbf{x}\right)\parallel=\infty.
\end{align}
Here, $\parallel\cdot\parallel$ defines a norm on $\mathbb{R}^n$.
\end{theorem}
\begin{remark}\label{re-1}
  \emph{Although Theorem \ref{th-14} is a consequence of Theorem \ref{th-16}, we must once again emphasize that our approach to prove Theorem \ref{th-14} is based on qualitative theory of dynamical systems, and is completely different from the proof of Theorem \ref{th-16}. We provide a very elementary dynamical proof.}
\end{remark}

The polynomial $R\left(x_1,\ldots,x_n\right)$ is \emph{quasi-homogeneous of weighted degree $l$ with respect to weight exponents $\mathbf{s}=\left(s_1,\ldots,s_n\right)$} if there exist positive integers $s_1,\ldots,s_n$ and $l$ such that for arbitrary $\lambda\in\mathbb{R}^+=\left\{\lambda\in \mathbb{R}, \lambda>0\right\}$, $R\left(\lambda^{s_1}x_1,\ldots,\lambda^{s_n}x_n\right)=\lambda^lR\left(x_1,\ldots,x_n\right)$. To each polynomial $R\left(x_1,\ldots,x_n\right)$, it can be written as the sum of its quasi-homogeneous parts $R=\sum_{i=0}^dR_i$, where  $R_i$ is quasi-homogeneous polynomial of weighted degree $i$ with respect to weight exponents $\mathbf{s}$. Moreover, the quasi-homogeneous term $R_d$ is called the higher quasi-homogeneous term of polynomial $R$ with respect to weight exponents $\mathbf{s}$,  and denote by $R_{\mathbf{s}}$. For a polynomial map $F=(f^1,\ldots,f^n): \mathbb{R}^n\rightarrow\mathbb{R}^n$, we denote $F_{\mathbf{s}}=\left(f_{\mathbf{s}}^1,\ldots,f_{\mathbf{s}}^n\right)$.

With the help of the algebraic skills, Cima et al. in \cite{MR1362759} proved the following theorem.
\begin{theorem}\label{th-13}
 Let $F=(f^1,\ldots,f^n): \mathbb{R}^n\rightarrow\mathbb{R}^n$ be a polynomial map with nowhere zero Jacobian determinant such that $F\left(0,\ldots,0\right)=\left(0,\ldots,0\right)$. If there is a weight  exponents $\mathbf{s}=\left(s_1,\ldots,s_n\right)\in\mathbb{N}_+^n$ such that $F_{\mathbf{s}}(\mathbf{x})=\mathbf{0}$ has only the trivial solution $\mathbf{x}=\mathbf{0}$, then $F$ is a global injective.
\end{theorem}

By Theorem \ref{th-16}, we give an alternate proof of Theorem \ref{th-13}. Our method to prove Theorem \ref{th-13} is completely different from \cite{MR1362759}.

Note that Theorem \ref{th-13}, Cima's result, is not a necessary and sufficient condition. In the above Example \ref{ex-1}, $F_{\mathbf{s}}$ for any  weight exponents $\mathbf{s}=\left(s_1,s_2\right)$ is always $\left(y^3,xy^2\right)$, which has nontrivial real solutions. This shows that Theorem \ref{th-16} improves Cima's result, i.e., Theorem \ref{th-13}.


The rest of this paper is organized as follows. In section \ref{se-2} we present some preliminary results. Section \ref{se-5} is devoted to prove Theorem \ref{th-3}. Theorem \ref{th-14} is proved in section \ref{se-6}. The proofs of Theorems \ref{th-16} and \ref{th-13} will be given in section \ref{se-4}.

\section{Preliminary results}\label{se-2}
In this section, we introduce some preliminary results.

%
%
%

\subsection{Limit sets}\label{sub-6}
Let $\varphi\left(t,p\right)$ be the integral curve of system \eqref{eq-3} passing through the point $p\in \mathbb{R}^2$ such that $\varphi\left(0,p\right)=p$. The set $K$ is \emph {positively invariant} if for each $p\in K$, $\varphi\left(t,p\right)\in K$ for all $t\geq0$.  We define the following sets
$$\omega\left(p\right)=\left\{x_0\in\mathbb{R}^2:\;\text{there exist}\;\left\{t_n\right\}\;\text{with}\;t_n\rightarrow \infty\;\text{and}\;\varphi\left(t_n,p\right)\rightarrow x_0\;\text{when}\;n\rightarrow\infty\right\}$$
and
$$\alpha\left(p\right)=\left\{x_0\in\mathbb{R}^2:\;\text{there exist}\;\left\{t_n\right\}\;\text{with}\;t_n\rightarrow -\infty\;\text{and}\;\varphi\left(t_n,p\right)\rightarrow x_0\;\text{when}\;n\rightarrow\infty\right\}.$$
The sets $\omega\left(p\right)$ and $\alpha\left(p\right)$ are called the $\omega$\emph{-limit set} and the $\alpha$\emph{-limit set} of $p$, respectively. Note that an $\alpha$-limit set of an integral curve $\varphi\left(t,p\right)$ is the $\omega$-limit set of the integral curve $\varphi\left(-t,p\right)$, i.e., after the time reversal. For this reason, it is sufficient to study the $\omega$-limit sets.

The following theorem characterized the structure of $\omega$-limit sets, see \cite{Dumortier2006qualitative} or \cite{MR2004534}.

\begin{theorem}[\emph{Poincar\'{e}-Bendixson Theorem}]\label{th-10}
Let $K$ be a positively invariant compact set for system \eqref{eq-3} containing a finite number of singular points, and $p\in K$. Then one of the following statements holds.
  \begin{enumerate}[(a)]
    \item $\omega\left(p\right)$ is a singular point.
    \item $\omega\left(p\right)$ is a periodic orbit.
    \item $\omega\left(p\right)$ consists of a finite number of singular points $p_1,\ldots,p_n$ and a finite number of orbits $\gamma_1,\ldots,\gamma_n$ such that $\alpha\left(\gamma_i\right)=p_i$, $\omega\left(\gamma_i\right)=p_{i+1}$ for $i=1,\ldots,n-1$, $\alpha\left(\gamma_n\right)=p_n$ and $\omega\left(\gamma_n\right)=p_1$. Possibly, some of the singular points $p_i$ are identified.
  \end{enumerate}
\end{theorem}
\subsection{Local structure of isolated singular points}\label{sub-7}
To characterize the local structure of isolated singular point $q$ of  analytic vector field  $\mathscr{X}$, we need the following definitions, see \cite{Dumortier2006qualitative,MR1175631,MR3642375} for more details. Let $U$ be a local region limited by two orbits of $\mathscr{X}$ inside with vertex singular point $q$.

The local region $U$ is a \emph{hyperbolic sector} of singular point $q$ if all orbits resemble hyperbolae in $U$, see $(1)$ of Figure \ref{Fig-3}.

The local region $U$ is an \emph{elliptic sector} of singular point $q$ if  all orbits have the singular point $q$ as both $\alpha$ and $\omega$ limit sets, see $(2)$ of Figure \ref{Fig-3}.

The local region $U$ is a \emph{parabolic sector} of singular point $q$ if all orbits positively (or all negatively) flow into singular point $q$, see $(3)$ of Figure \ref{Fig-3}.

Note that the two boundaries of a hyperbolic sector are \emph{separatrices} of singular point $q$.

\begin{figure}[H]
\centering
\begin{minipage}{0.225\linewidth}
\centering
\psfrag{0}{$q$}
\centerline{\includegraphics[width=1\textwidth]{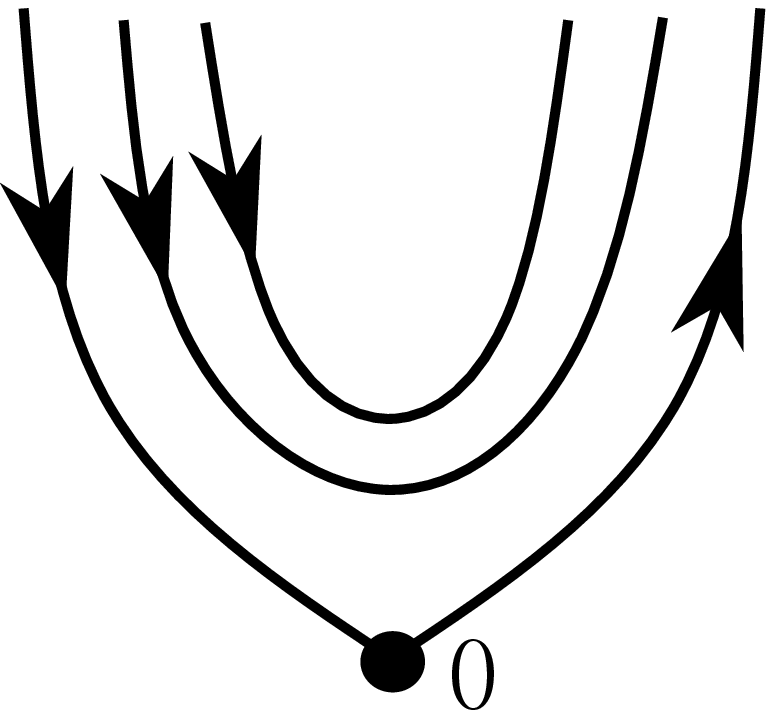}}
\centerline{(1)\;Hyperbolic sector}
\end{minipage}
\qquad
\begin{minipage}{0.225\linewidth}
\centering
\psfrag{0}{$q$}
\centerline{\includegraphics[width=1\textwidth]{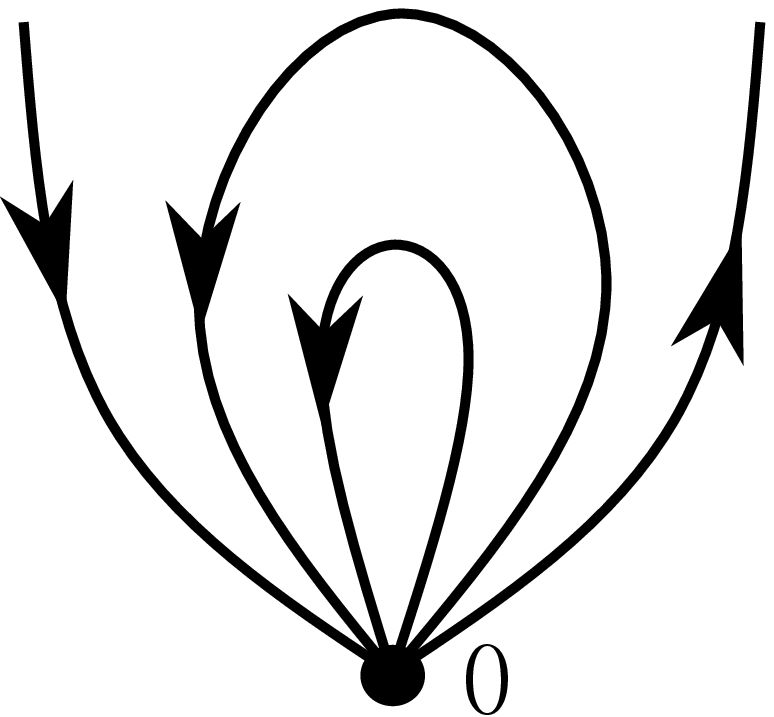}}
\centerline{(2)\;Elliptic sector}
\end{minipage}
\qquad
\begin{minipage}{0.225\linewidth}
\centering
\psfrag{0}{$q$}
\centerline{\includegraphics[width=1\textwidth]{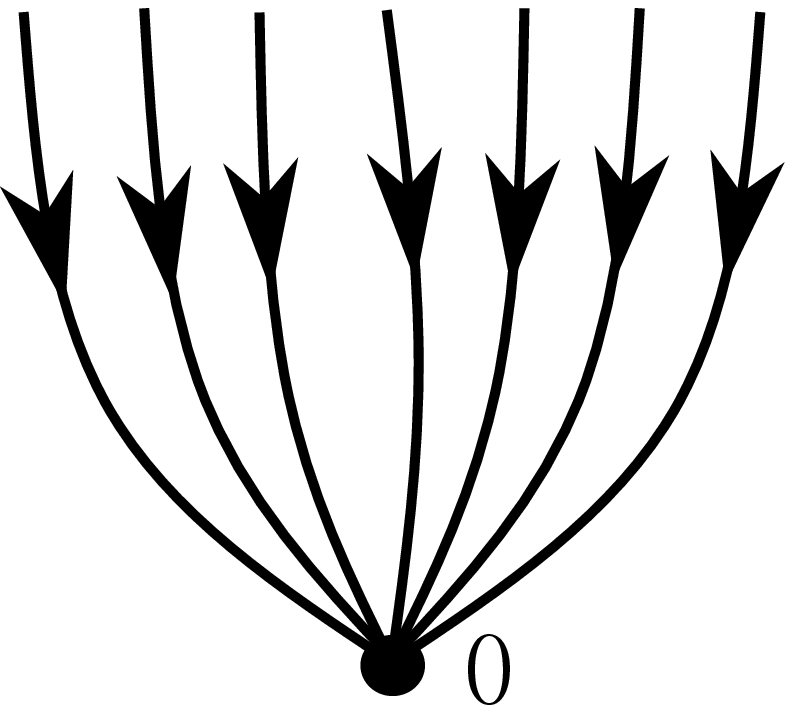}}
\centerline{(3)\;Parabolic sector}
\end{minipage}
\caption{Sectors near a singular point $q$.}\label{Fig-3}
\end{figure}

The next theorem is proved in \cite{MR0153903,MR1929104}.
\begin{theorem}\label{th-6}
If $q$ is an isolated singular point of analytic vector field $\mathscr{X}$, then singular point $q$ can only be a center, a focus, or be decomposed into a finite number of hyperbolic, elliptic and parabolic sectors.
\end{theorem}

\subsection{Compactification of vector field}\label{sub-1}
If $P\left(x,y\right),Q\left(x,y\right)$ are real polynomials in variables $x$ and $y$, then system \eqref{eq-3} is  polynomial system. We say that system \eqref{eq-3} has degree $d$ if $d=\max\{\text{deg}P,\text{deg}Q\}$. In order to investigate the behavior of the trajectories of a planar polynomial vector field $\mathscr{X}$ near infinity, we need to compactify it. Usually there are two methods to compactify a planar polynomial vector field: \emph{Pioncar\'{e} compactification} and \emph{Bendixson compactification}. Next we introduce these two compactifications, see Chapter $13$ of \cite{MR0350126} or Chapter $5$ of \cite{Dumortier2006qualitative} for more details.
%

We first briefly describe the Pioncar\'{e} compactification. The (Poincar\'{e}) unit sphere is tangent to $xy$-plane at the origin $(0,0)$ in  Figure \ref{Fig-6}.  The point $M(x,y)$ in the $xy$-plane connects with the center $O$ of the sphere through a straight line which intersects the sphere at the two points $M'$ and $\overline{M'}$. We project the point $M'$ on the lower hemisphere vertically in to the $xy$-plane which  leads to the point $M''$ on the \emph{Poincar\'{e} disk} $\mathbb{D}^2=\left\{(x,y)\big|x^2+y^2\leq1\right\}$. It is known that the boundary of the disc $\mathbb{D}^2$, i.e. the unit circle  $\mathbb{S}^1=\left\{(x,y)\big|x^2+y^2=1\right\}$, corresponds to the infinity of $\mathbb{R}^2$, and called the \emph{equator}. Then the vector field $\mathscr{X}$ can be extended analytically to  Poincar\'{e} sphere by the central projection. So the global dynamics of $\mathscr{X}$ can be characterized on the Poincar\'{e} disk, that is, the finite and  infinity  of $\mathscr{X}$ respectively corresponding the interior and boundary $\mathbb{S}^1$ of $\mathbb{D}^2$.

\begin{figure}[H]
\tiny
\centering
\begin{minipage}{0.5\linewidth}
\centering
\psfrag{0}{Equator $\mathbb{S}^1$}
\psfrag{1}{$x$}
\psfrag{2}{$y$}
\psfrag{3}{$M(x,y)$}
\psfrag{4}{$M'$}
\psfrag{5}{Poincar\'{e} disc $\mathbb{D}^2$}
\psfrag{6}{$O$}
\psfrag{7}{$\overline{M'}$}
\centerline{\includegraphics[width=1\textwidth]{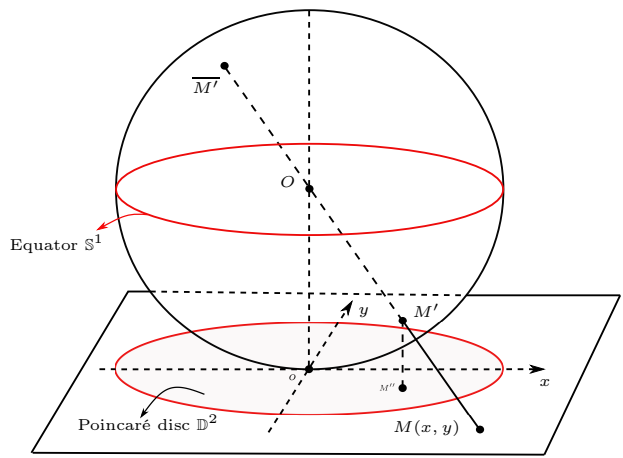}}
\end{minipage}
\caption{The Pioncar\'{e} compactification and Poincar\'{e} disc.}\label{Fig-6}
\end{figure}

Roughly speaking, the construction of the Bendixson compactification is as follows. Let $\mathbb{S}^2=\{Y=\left(y_1,y_2,y_3\right)\in\mathbb{R}^3:y_1^{2}+y_2^{2}+y_3^{2}=1/4\}$ (the \emph{Bendixson sphere}). Assume that $\mathscr{X}$ is defined with the tangent plane to the sphere $\mathbb{S}^2$ at the south pole $S=\left(0,0,-1/2\right)$, that is, $xy$-plane, see Figure \ref{Fig-1}. The \emph{Bendixson compactified vector field $b\left(\mathscr{X}\right)$ associated to $\mathscr{X}$} is an analytic vector field induced on $\mathbb{S}^2$ by the stereographic projection. More precisely, consider the stereographic projection $p_N$ from the north pole $N=\left(0,0,1/2\right)$ to the $xy$-plane. Thus the vector field $\mathscr{X}$ can be induced to $\mathbb{S}^2\setminus N$ by the map $p_N^{-1}$. Obviously, the infinity of the $xy$-plane is transformed by $p_N^{-1}$ into the north pole $N$.

To simplify the calculations, we take the two local charts on the Bendixson sphere $\mathbb{S}^2$ given by $$U_N=\mathbb{S}^2\setminus N,\quad U_S=\mathbb{S}^2\setminus S$$
with associated local maps
$$p_N:U_N\rightarrow\mathbb{R}^2,\quad p_S:U_S\rightarrow\mathbb{R}^2,$$
where $p_S$ is the stereographic projection of $\mathbb{S}^2$ from the south pole $S$ to the $uv$-plane  given by the equation  $y_3=1/2$. The map $p_S\circ p_N^{-1}$ from the $xy$-plane minus $S$ to the $uv$-plane minus $N$ is given by
\begin{align}\label{eq-4}
&u=\frac{x}{x^2+y^2},\quad v=\frac{y}{x^2+y^2}.
\end{align}
\begin{figure}[H]
\centering
\begin{minipage}{0.5\linewidth}
\centering
\psfrag{0}{$O$}
\psfrag{1}{$y_1$}
\psfrag{2}{$y_2$}
\psfrag{3}{$y_3$}
\psfrag{4}{$N$}
\psfrag{5}{$S$}
\psfrag{6}{$xy$-plane}
\psfrag{7}{$uv$-plane}
\centerline{\includegraphics[width=1\textwidth]{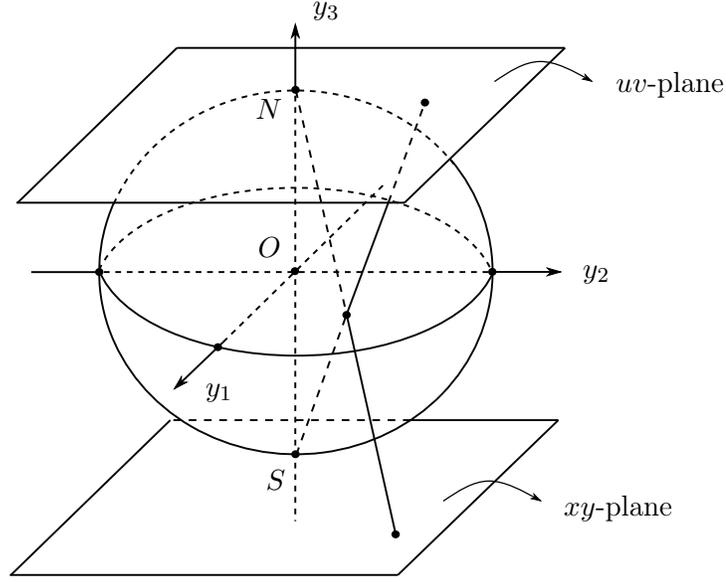}}
\end{minipage}
\caption{The stereographic projection and \emph{Bendixson sphere}.}\label{Fig-1}
\end{figure}
After a scaling of the independent variable in the local chart $\left(U_N, p_N\right)$ the expression for $b\left(\mathscr{X}\right)$ is
\begin{align}\label{eq-5}
\begin{cases}
\dot{u}=\left(u^2+v^2\right)^d\left[\left(v^2-u^2\right)P\left(\dfrac{u}{u^2+v^2},\dfrac{v}{u^2+v^2}\right)-2uvQ\left(\dfrac{u}{u^2+v^2},\dfrac{v}{u^2+v^2}\right)\right],\\
\specialrule{0em}{3pt}{3pt}
\dot{v}=\left(u^2+v^2\right)^d\left[\left(u^2-v^2\right)Q\left(\dfrac{u}{u^2+v^2},\dfrac{v}{u^2+v^2}\right)-2uvP\left(\dfrac{u}{u^2+v^2},\dfrac{v}{u^2+v^2}\right)\right].
\end{cases}
\end{align}
Hence the infinity of system \eqref{eq-3}  becomes the origin of  system \eqref{eq-5}.
\begin{remark}\label{re-3}
 By the Bendixson compactification, the infinity of system \eqref{eq-3} on the Poincar\'{e} disk, i.e the equator $\mathbb{S}^1$, is transformed into the origin of  system \eqref{eq-5}. We say in what follows, unless otherwise specified, that the infinity of a vector field is refer to the equator $\mathbb{S}^1$  on the Poincar\'{e} disk.
\end{remark}
\subsection{Topological index}\label{sub-3}
The following results are well known, see Chapter 6 of \cite{Dumortier2006qualitative} or \cite{MR1175631}.
\begin{theorem}[\emph{Poincar\'{e} Index Formula}]\label{th-7}
Let $q$  be an isolated singular point having the finite sectorial decomposition property. Let $e$, $h$ and $p$ denote the number of elliptic, hyperbolic and parabolic sectors of $q$, respectively. Then the index of $q$ is $\left(e-h\right)/2+1$.
\end{theorem}


\begin{proposition}\label{pr-2}
If a vector field $\mathscr{X}$ has only isolated singular point, then the sum of indices singular points in a region $D$ enclosed by its any  periodic orbit is $1$.
\end{proposition}

The next proposition follows immediately from Theorem \ref{th-7}.
\begin{proposition}\label{pr-3}
For an analytic vector field $\mathscr{X}$, the index of a monodromic singular point (i.e., a center or a focus) is $1$.
\end{proposition}

\section{Proof of Theorem \ref{th-3}}\label{se-5}
The main purpose of this section is to prove Theorem \ref{th-3}.
\begin{lemma}\label{le-1}
 Let $F=\left(f,g\right):\mathbb{R}^2\rightarrow \mathbb{R}^2$ be a polynomial map with nowhere zero Jacobian determinant such that $F\left(0,0\right)=\left(0,0\right)$.  If $\emph{gcd}(f,g)$ is the greatest common divisor of $f$ and $g$, then $y\nmid\emph{gcd}(f,g)$ and $x\nmid\emph{gcd}(f,g)$.
\end{lemma}

\begin{proof}
Suppose that $y\mid\text{gcd}(f,g)$.  Then there exist two polynomials $\bar{f}_1\left(x,y\right)$ and $\bar{g}_1\left(x,y\right)$ such that $f=y\bar{f}_1$ and $g=y\bar{g}_1$. This is in contraction with $\text{det}DF\left(0,0\right)\neq0$. Thus, $y\nmid\text{gcd}(f,g)$. By a similar way, one can prove that $x\nmid\text{gcd}(f,g)$ also holds.
\end{proof}
The following theorem is due to Mazzi and Sabatini \cite{MR969422}.
\begin{theorem}\label{th-15}
Assume that system \eqref{eq-3} is $C^k$ with an isolated singular point $q$ and $k\in\mathbb{N}^+\cup\{\infty\}$. Then the singular point $q$ is a center if and only if there exists a $C^k$ first integral with an isolated minimum at $q$.
\end{theorem}

\begin{proposition}\label{pr-1}
  Let $F=\left(f,g\right):\mathbb{R}^2\rightarrow \mathbb{R}^2$ be a polynomial map with nowhere zero Jacobian determinant such that $F\left(0,0\right)=\left(0,0\right)$. Then $q$ is a singular point of the Hamiltonian vector field \eqref{eq-1} if and only if $F\left(q\right)=\left(0,0\right)$. In this case, this singular point $q$ is a center of vector field \eqref{eq-1}.
\end{proposition}
\begin{proof}
The Hamiltonian vector field $\mathcal{X}$ can be written as
\begin{align*}
&\left(
\begin{array}{l}
\dot{x}\\
\dot{y}\\
\end{array}
\right)=
 \left(
\begin{array}{cc}
-f_y&-g_y \\
f_x&g_x\\
\end{array}
\right)
 \left(
\begin{array}{c}
f\\
g\\
\end{array}
\right).
\end{align*}
 Indeed, $q$ is a singular point of $\mathcal{X}$ if and only if
 \begin{align*}
&\left(
\begin{array}{cc}
-f_y\left(q\right)&-g_y\left(q\right) \\
f_x\left(q\right)&g_x\left(q\right)\\
\end{array}
\right)
 \left(
\begin{array}{c}
f\left(q\right)\\
g\left(q\right)\\
\end{array}
\right)=\left(
\begin{array}{l}
0\\
0\\
\end{array}
\right).
\end{align*}
The sufficiency is obvious. Since $\text{det D} F\left(q\right)\neq0$, $f\left(q\right)=g\left(q\right)=0$. The necessity holds.

Since $\text{det D} F\left(q\right)\neq0$, singular point $q$ is an isolated. The Hamiltonian vector field \eqref{eq-1} has the Hamiltonian
\begin{align}\label{eq-34}
  &H\left(x,y\right)=\frac{f^2\left(x,y\right)+g^2\left(x,y\right)}{2}.
\end{align}
Clearly, $q$ is an isolated minimum of $H$, that is, $H\left(x,y\right)\geq H\left(q\right)=0$. By Theorem \ref{th-15}, $q$ is a center. This proves the Proposition \ref{pr-1}.
\end{proof}
%
%

\begin{lemma}\label{le-8}
  Let $f\left(x,y\right)=\sum\nolimits_{i=1}^{n}f_i\left(x,y\right)$, $g\left(x,y\right)=\sum\nolimits_{j=1}^{m}g_j\left(x,y\right)$ and $F=\left(f,g\right)$ with nowhere zero Jacobian determinant $\det DF\left(x,y\right)$, where $f_i\left(x,y\right)$ and $g_j\left(x,y\right)$ are homogeneous polynomials of degree $i$ and $j$, respectively. Then $f_1\left(x,y\right)=g_1\left(x,y\right)=0$ if and only if $\left(x,y\right)=\left(0,0\right)$.
\end{lemma}
\begin{proof}
Since $f_1\left(x,y\right)$ and $g_1\left(x,y\right)=0$ are linear, we have
\begin{align}\label{eq-35}
&\left(
\begin{array}{l}
f_1(x,y)\\
g_1(x,y)\\
\end{array}
\right)=
 \left(
\begin{array}{cc}
f_{1x}& f_{1y}\\
g_{1x}&g_{1y}\\
\end{array}
\right)
 \left(
\begin{array}{c}
x\\
y\\
\end{array}
\right)=
 \left(
\begin{array}{c}
0\\
0\\
\end{array}
\right).
\end{align}
The determinant of the coefficient matrix of system \eqref{eq-35} is $\det DF\left(0,0\right)\neq0$. The lemma holds.
\end{proof}

\begin{lemma}\label{le-5}
Let $\phi\left(t,p\right)$ be the trajectory of  vector field \eqref{eq-1}  passing through the regular point $p\in\mathbb{R}^2$ for $t\in\mathbb{R}$.  Then one of the following statements holds.
\begin{enumerate}[(a)]
  \item $\omega\left(p\right)=\alpha\left(p\right)$ is a periodic orbit located at a period annulus.
  \item $\omega\left(p\right)\subset \mathbb{S}^1$ and $\alpha\left(p\right)\subset\mathbb{S}^1$.
\end{enumerate}
 Here $\mathbb{S}^1$ is the infinity of the Poincar\'{e} disc.
\end{lemma}
\begin{proof}
By Proposition \ref{pr-1}, each finite singular point $q$ of  vector field \eqref{eq-1} is a center. From Theorem \ref{th-10} it follows that $\omega\left(p\right)=\alpha\left(p\right)$ is a periodic orbit, or $\omega\left(p\right)\subset \mathbb{S}^1$ and $\alpha\left(p\right)\subset\mathbb{S}^1$.
\end{proof}

The limit sets of vector field $b\left(\mathcal{X}\right)$ are given in next proposition.
\begin{proposition}\label{pr-4}
  Let $\psi\left(t,\bar{p}\right)$ be the trajectory of vector field $b\left(\mathcal{X}\right)$   passing through the regular point $\bar{p}\in\mathbb{R}^2$ for $t\in\mathbb{R}$.  Then one of the following statements holds.
\begin{enumerate}[(a)]
  \item $\omega\left(\bar{p}\right)=\alpha\left(\bar{p}\right)$ is a periodic orbit located at a period annulus.
  \item $\omega\left(\bar{p}\right)=\alpha\left(\bar{p}\right)$ is the origin of vector field $b\left(\mathcal{X}\right)$ .
\end{enumerate}
\end{proposition}
\begin{proof}
By Lemma \ref{le-5}, on the Bendixson sphere $\mathbb{S}^2$ (see Figure \ref{Fig-1}) the $\alpha$ and $\omega$ limit sets of each orbit
of the vector field $b\left(\mathcal{X}\right)$  can only be the north pole $N$  or a periodic orbit located at a period annulus. This proposition is confirmed.
\end{proof}

Let $f\left(x,y\right)=\sum\nolimits_{i=1}^{n}f_i\left(x,y\right)$, $g\left(x,y\right)=\sum\nolimits_{j=1}^{m}g_j\left(x,y\right)$ and $d=\max\{n,m\}$, where $f_i\left(x,y\right)$ and $g_j\left(x,y\right)$ are homogeneous polynomials of degree $i$ and $j$, respectively. From the equation \eqref{eq-5}, the Bendixson compactification of vector field $\mathcal{X}$ is given by
\begin{align}\label{eq-6}
\begin{cases}
\dot{u}=&\sum\limits_{i=1}^d\sum\limits_{j=1}^d\left(u^2+v^2\right)^{2d-i-j}[\left(u^2-v^2\right)\left(f_i\left(u,v\right)f_{jy}\left(u,v\right)
+g_i\left(u,v\right)g_{jy}\left(u,v\right)\right)\\
&-2uv\left(f_i\left(u,v\right)f_{jx}\left(u,v\right)+g_i\left(u,v\right)g_{jx}\left(u,v\right)\right)],\\
\specialrule{0em}{3pt}{3pt}
\dot{v}=&\sum\limits_{i=1}^d\sum\limits_{j=1}^d\left(u^2+v^2\right)^{2d-i-j}[\left(u^2-v^2\right)\left(f_i\left(u,v\right)f_{jx}\left(u,v\right)
+g_i\left(u,v\right)g_{jx}\left(u,v\right)\right)\\
&+2uv\left(f_i\left(u,v\right)f_{jy}\left(u,v\right)+g_i\left(u,v\right)g_{jy}\left(u,v\right)\right)].
\end{cases}
\end{align}
It is easy to check that the origin of system \eqref{eq-6} is degenerate.

The local dynamical behavior of vector field $b\left(\mathcal{X}\right)$ on the Poincar\'{e} disk is given as follows.
\begin{proposition}\label{pr-5}
  Let $F=\left(f,g\right):\mathbb{R}^2\rightarrow \mathbb{R}^2$ be a polynomial map such that \emph{det} $DF\left(x,y\right)$ is nowhere zero and $F\left(0,0\right)=\left(0,0\right)$. Then the following statements hold.
  \begin{enumerate}[(a)]
    \item The finite singular points of polynomial vector field $b\left(\mathcal{X}\right)$ other than its origin are centers.
    \item For polynomial vector field $b\left(\mathcal{X}\right)$, there are no infinite singular points on the Poincar\'{e} disc.
    \item  The origin of the polynomial vector field $b\left(\mathcal{X}\right)$ has no parabolic sectors.
  \end{enumerate}
\end{proposition}
\begin{proof}
$(a)$ By Proposition \ref{pr-1}, statement $(a)$ can be easily proved.

$(b)$ Taking the Poincar\'{e} transformation $u=1/\bar{v}$, $v=\bar{u}/\bar{v}$ and rescaling the time $dt=\bar{v}^{4d-1}d\tau$, system \eqref{eq-6} can be written as
\begin{align}\label{eq-9}
\begin{cases}
\dot{\bar{u}}=&\left(1+\bar{u}^2\right)\sum\limits_{i=1}^d\sum\limits_{j=1}^dj\bar{v}^{i+j-2}\left(1+\bar{u}^2\right)^{2d-i-j}\left(f_i\left(1,\bar{u}\right)f_{j}
\left(1,\bar{u}\right)+g_i\left(1,\bar{u}\right)g_{j}\left(1,\bar{u}\right)\right),\\
\specialrule{0em}{3pt}{3pt}
\dot{\bar{v}}=&-\bar{v}\sum\limits_{i=1}^d\sum\limits_{j=1}^d\bar{v}^{i+j-2}\left(1+\bar{u}^2\right)^{2d-i-j}[\left(1-\bar{u}^2\right)\left(f_i\left(1,\bar{u}\right)f_{jy}\left(1,\bar{u}\right)
+g_i\left(1,\bar{u}\right)g_{jy}\left(1,\bar{u}\right)\right)\\
&-2\bar{u}\left(f_i\left(1,\bar{u}\right)f_{jx}\left(1,\bar{u}\right)+g_i\left(1,\bar{u}\right)g_{jx}\left(1,\bar{u}\right)\right)].\\
\end{cases}
\end{align}
Applying Lemma \ref{le-8}, we have that the system \eqref{eq-9} has no singular points on the $\bar{u}$-axis.

Using the Poincar\'{e} transformation $u=\bar{u}/\bar{v}$, $v=1/\bar{v}$ with the scaling $dt=\bar{v}^{4d-1}d\tau$, system \eqref{eq-6} becomes
\begin{align}\label{eq-11}
\begin{cases}
\dot{\bar{u}}=&-\left(1+\bar{u}^2\right)\sum\limits_{i=1}^d\sum\limits_{j=1}^dj\bar{v}^{i+j-2}\left(1+\bar{u}^2\right)^{2d-i-j}\left(f_i\left(\bar{u},1\right)f_{j}
\left(\bar{u},1\right)+g_i\left(\bar{u},1\right)g_{j}\left(\bar{u},1\right)\right),\\
\specialrule{0em}{3pt}{3pt}
\dot{\bar{v}}=&-\bar{v}\sum\limits_{i=1}^d\sum\limits_{j=1}^d\bar{v}^{i+j-2}\left(1+\bar{u}^2\right)^{2d-i-j}[\left(\bar{u}^2-1\right)\left(f_i\left(\bar{u},1\right)
f_{jx}\left(\bar{u},1\right)+g_i\left(\bar{u},1\right)g_{jx}\left(\bar{u},1\right)\right)\\
&+2\bar{u}\left(f_i\left(\bar{u},1\right)f_{jy}\left(\bar{u},1\right)+g_i\left(\bar{u},1\right)g_{jy}\left(\bar{u},1\right)\right)].
\end{cases}
\end{align}
Similarly, the origin is not a singular point of system \eqref{eq-11}.  The statement $(b)$ holds.

$(c)$ Assume that the origin of $b\left(\mathcal{X}\right)$ has a attracting parabolic sector $V$. Let $\psi\left(t,\bar{p}\right)$ be the trajectory of vector field $b\left(\mathcal{X}\right)$  passing through the regular point $\bar{p}\in V$. Then $\omega\left(\bar{p}\right)$ is the origin of $b\left(\mathcal{X}\right)$. From Proposition \ref{pr-4}, $\alpha\left(\bar{p}\right)$ is also the origin of $b\left(\mathcal{X}\right)$, which means that $V$ is an elliptic sector. This is a contradiction. So statement $(c)$ is confirmed.
\end{proof}

\begin{proof}[{\bf Proof of Theorem \ref{th-3}}]
For clarity, we will split the proof into three steps.

Firstly, we prove that $(a)\Rightarrow (b) \Rightarrow (c)\Rightarrow (d)$. Using Theorem \ref{th-1}, it is easy to prove that $(a)\Rightarrow (b)$. By the definition of monodromic singular point, it is obvious that $(b) \Rightarrow (c)\Rightarrow (d)$.

Secondly, we show that $(d) \Rightarrow (c) \Rightarrow (b) \Rightarrow (a)$ as follows.

$(d)\Rightarrow(c)$. By Theorem \ref{th-7}, the index of the origin of the polynomial vector field $b\left(\mathcal{X}\right)$ is $e/2+1$. Since the origin of vector field $\mathcal{X}$ is a center, on the Bendixson sphere $\mathbb{S}^2$ (see Figure \ref{Fig-1}) the south pole $S$ is also a center. This means that there exists a periodic orbit of $b\left(\mathcal{X}\right)$  such that it contains all finite singular points of $b\left(\mathcal{X}\right)$. Let $\bar{c}$ denote the number of finite singular points of $b\left(\mathcal{X}\right)$ other than its origin. By statement $(a)$ of Proposition \ref{pr-5} and Proposition \ref{pr-2}, we have $e/2+1+\bar{c}=1$, that is, $e=\bar{c}=0$. So, the origin is the unique finite singular point of $b\left(\mathcal{X}\right)$. From statement $(c)$ of Proposition \ref{pr-5} and Theorem \ref{th-6} it follows that the origin is a monodromic singular point.

$(c)\Rightarrow (b)$. I'lyashenko in \cite{MR1133882}  and \'{E}calle in \cite{MR1399559} prove that a monodromic singular point of an analytic vector field must be either a center or a focus. Assume that the origin of $b\left(\mathcal{X}\right)$  is a attracting focus. Let $U$ be a sufficiently small neighborhood  of the origin.  Then there exists a orbit $\psi\left(t,\bar{p}\right)$ of $b\left(\mathcal{X}\right)$ passing through the regular point $\bar{p}\in U\setminus\{O\}$ such that $\lim_{t\rightarrow+\infty}\psi\left(t,\bar{p}\right)=O$, that is, $\omega\left(\bar{p}\right)=\{O\}$. By Proposition \ref{pr-4},
$\alpha\left(\bar{p}\right)=\{O\}$ which is a contradiction. Thus the origin of $b\left(\mathcal{X}\right)$ is a center.

$(b)\Rightarrow(a).$ The origin of the polynomial vector field $b\left(\mathcal{X}\right)$ is a center, which means that every solution trajectory of the polynomial vector field $\mathcal{X}$ in a neighborhood of infinity (the equator $\mathbb{S}^1$) is a closed orbit. Thus all finite singular points of $\mathcal{X}$ are contained in a periodic orbit. From Propositions \ref{pr-2}, \ref{pr-3} and \ref{pr-1}, it follows that the $(0,0)$ is the unique finite singular point of $\mathcal{X}$, which is a center. Combining with the behavior of the trajectories of $\mathcal{X}$ near infinity, we obtain that $(0,0)$ is a global center of the vector field $\mathcal{X}$. Using Theorem \ref{th-1},  $F$ is a global diffeomorphism of the plane onto itself.

Finally, it follows from Theorem \ref{th-15} that statement $(e)$ is equivalent to statement $(b)$, that is, $(b) \Leftrightarrow (e)$.

We complete the proof of Theorem \ref{th-3}.
\end{proof}

\section{Proof of Theorem \ref{th-14}}\label{se-6}
Let $l$ be a straight line and a point $E\in l$. The point $E$ is a \emph{contact point} of the straight line $l$ with a vector field \eqref{eq-3} if the vector $\mathscr{X}\left(E\right)$ is parallel to $l$. To prove Theorem \ref{th-14}, we need the following lemmas.
\begin{lemma}\label{le-2}
Assume that vector field \eqref{eq-3} is a polynomial vector field with $\bar{d}=\max\left\{\emph{deg} P,\emph{deg} Q\right\}$.  Let $l$ be a straight line. Then either $l$ is an invariant straight line of $\mathscr{X}$, or $\mathscr{X}$ has at most $\bar{d}$ contact points (including
the  singular points) along $l$.
\end{lemma}

The proof of Lemma \ref{le-2} can be found in Lemma 8.12 of \cite{Dumortier2006qualitative}.
\begin{lemma}\label{le-3}
Let $G_1\left(x,y\right)$  and $G_2\left(x,y\right)$ be two functions defined on $\mathbb{R}^2$. Then the limit
\begin{align}\label{eq-41}
&\lim_{\left(x,y\right)\rightarrow\left(0,0\right)}\frac{1}{G_1^2\left(\dfrac{x}{x^2+y^2},\dfrac{y}{x^2+y^2}\right)+G_2^2\left(\dfrac{x}{x^2+y^2},\dfrac{y}{x^2+y^2}\right)}
\end{align}
exists if and only if the limit
\begin{align}\label{eq-42}
&\lim_{|x|+|y|\rightarrow+\infty}\frac{1}{G_1^2\left(x,y\right)+G_2^2\left(x,y\right)}
\end{align}
exists. Moreover, these two limits are equal.
\end{lemma}
\begin{proof}
\emph {Sufficiency.} Assume that
\begin{align*}
&\lim_{|x|+|y|\rightarrow+\infty}\frac{1}{G_1^2\left(x,y\right)+G_2^2\left(x,y\right)}=\mathcal{L}.
\end{align*}
For every $\varepsilon>0$, there exists $M>0$ such that if $|x|+|y|>M$, then
\begin{align}\label{eq-43}
&\left|\frac{1}{G_1^2\left(x,y\right)+G_2^2\left(x,y\right)}-\mathcal{L}\right|<\varepsilon.
\end{align}

Let
$$\left(x,y\right)=\left(\frac{X}{X^2+Y^2},\frac{Y}{X^2+Y^2}\right).$$
Then
$$\left|X\right|+\left|Y\right|\leq\frac{2\left(X^2+Y^2\right)}{\left|X\right|+\left|Y\right|}=\frac{2}{|x|+|y|}.$$
Taking $\delta=2/M$, we have that if $0<|X|+|Y|<\delta$, then
\begin{align}\label{eq-44}
&\left|\frac{1}{G_1^2\left(\frac{X}{X^2+Y^2},\frac{Y}{X^2+Y^2}\right)+G_2^2\left(\frac{X}{X^2+Y^2},\frac{Y}{X^2+Y^2}\right)}-\mathcal{L}\right|<\varepsilon
\end{align}
for every $\varepsilon>0$. Consequently, the limit \eqref{eq-41} exists.

\emph {Necessity.}  Assume that
\begin{align*}
&\lim_{\left(x,y\right)\rightarrow\left(0,0\right)}\frac{1}{G_1^2\left(\dfrac{x}{x^2+y^2},\dfrac{y}{x^2+y^2}\right)+G_2^2\left(\dfrac{x}{x^2+y^2},\dfrac{y}{x^2+y^2}\right)}=\mathcal{L}_1.
\end{align*}
For every $\varepsilon>0$, there exists $\delta_1<0$ such that if $0<|x|+|y|<\delta_1$, then
\begin{align}\label{eq-45}
&\left|\frac{1}{G_1^2\left(\dfrac{x}{x^2+y^2},\dfrac{y}{x^2+y^2}\right)+G_2^2\left(\dfrac{x}{x^2+y^2},\dfrac{y}{x^2+y^2}\right)}-\mathcal{L}_1\right|<\varepsilon.
\end{align}
Setting
$$\left(X_1,Y_1\right)=\left(\frac{x}{x^2+y^2},\frac{y}{x^2+y^2}\right),$$
one can obtain that
$$\left|X_1\right|+\left|Y_1\right|\geq \frac{X_1^2+Y_1^2}{\left|X_1\right|+\left|Y_1\right|}=\frac{1}{|x|+|y|}.$$
Taking $M_1=1/\delta_1$, we get that if $|X_1|+|Y_1|>M_1$, then
\begin{align}\label{eq-46}
&\left|\frac{1}{G_1^2\left(X_1,Y_1\right)+G_2^2\left(X_1,Y_1\right)}-\mathcal{L}_1\right|<\varepsilon
\end{align}
for every $\varepsilon>0$. So, the limit \eqref{eq-42} exists.

This lemma is confirmed.
\end{proof}

The Lemma \ref{le-3} tells us that Theorem \ref{th-14} is equivalent to the following proposition.
\begin{proposition}\label{pr-6}
 Let $F=\left(f,g\right):\mathbb{R}^2\rightarrow \mathbb{R}^2$ be a polynomial map with nowhere zero Jacobian determinant such that $F\left(0,0\right)=\left(0,0\right)$. Then $F$ is a global injective if and only if
\begin{align}\label{eq-48}
&\lim_{\left(x,y\right)\rightarrow\left(0,0\right)}\bar{I}\left(x,y\right)=
\lim_{\left(x,y\right)\rightarrow\left(0,0\right)}\frac{1}{f^2\left(\dfrac{x}{x^2+y^2},\dfrac{y}{x^2+y^2}\right)+g^2\left(\dfrac{x}{x^2+y^2},\dfrac{y}{x^2+y^2}\right)}
\end{align}
exists.
\end{proposition}

\begin{proof}
\emph {Sufficiency.} Let $f\left(x,y\right)=\sum\nolimits_{i=1}^{n}f_i\left(x,y\right)$, $g\left(x,y\right)=\sum\nolimits_{j=1}^{m}g_j\left(x,y\right)$ and $d=\max\{n,m\}$, where $f_i\left(x,y\right)$ and $g_j\left(x,y\right)$ are homogeneous polynomials with degree $i$ and $j$, respectively. Then there exists a $\theta_0\in\left[0,2\pi\right]$ such that $f_d^2\left(\cos\theta_0,\sin\theta_0\right)+g_d^2\left(\cos\theta_0,\sin\theta_0\right)\neq0$. Otherwise, $f_d\left(x,y\right)=g_d\left(x,y\right)\equiv0$. We have
\begin{equation}\label{eq-38}
\footnotesize
 \begin{split}
  \lim_{r\rightarrow 0}\bar{I}\left(r\cos\theta_0,r\sin\theta_0\right)&=\lim_{r\rightarrow 0}\frac{1}{f^2\left(\frac{\cos\theta_0}{r},\frac{\sin\theta_0}{r}\right)+g^2\left(\frac{\cos\theta_0}{r},\frac{\sin\theta_0}{r}\right)}\\
&=\lim_{r\rightarrow0}\frac{r^{2d}}{\sum\limits_{i=1}^d\sum\limits_{j=1}^dr^{2d-i-j}\left[f_i\left(\cos\theta_0,\sin\theta_0\right)
f_{j}\left(\cos\theta_0,\sin\theta_0\right)+g_i\left(\cos\theta_0,\sin\theta_0\right)
g_{j}\left(\cos\theta_0,\sin\theta_0\right)\right]}\\
&=0.
 \end{split}
\end{equation}
Therefore, $\bar{I}(x,y)\rightarrow 0$ as $(x,y)\rightarrow(0,0)$ along the straight line $$l:=\left\{(x,y):x=r\cos\theta_0,y=r\sin\theta_0,r\in\mathbb{R}\right\}.$$
Since the limit $\lim_{(x,y)\rightarrow(0,0)}\bar{I}(x,y)$ exists, we obtain
\begin{align}\label{eq-10}
&\lim_{(x,y)\rightarrow(0,0)}\bar{I}(x,y)=\lim_{
\begin{matrix}
(x,y)\rightarrow(0,0)\\
\text{\small Along straight line}\;l
\end{matrix}
}\bar{I}(x,y)=0.
\end{align}

 The Hamiltonian vector field \eqref{eq-1} has the Hamiltonian
\begin{align}\label{eq-2}
  &H\left(x,y\right)=\frac{f^2\left(x,y\right)+g^2\left(x,y\right)}{2}.
\end{align}
By Proposition 1.2 of \cite{MR3642375}, the vector field $b\left(\mathcal{X}\right)$  has the first integral
$$\bar{I}\left(u,v\right)=\frac{1}{f^2\left(\dfrac{u}{u^2+v^2},\dfrac{v}{u^2+v^2}\right)+g^2\left(\dfrac{u}{u^2+v^2},\dfrac{v}{u^2+v^2}\right)}$$
with $u^2+v^2\neq0$. It is easy to check that
\begin{align}\label{eq-26}
&\tilde{I}\left(u,v\right)=
\begin{cases}
\bar{I}\left(u,v\right),\quad \left(u,v\right)\neq\left(0,0\right),\\
  0,\quad \left(u,v\right)=\left(0,0\right),
\end{cases}
\end{align}
is also first integral of vector field $b\left(\mathcal{X}\right)$. Since
$$\lim_{(u,v)\rightarrow(0,0)}\bar{I}(u,v)=0,$$
$\tilde{I}\left(u,v\right)$ is continuous on $\mathbb{R}^2$.

Let $\gamma(t)=\left(u(t),v(t)\right)$ be an orbit of $b\left(\mathcal{X}\right)$ tending to the origin as $t\rightarrow+\infty$ (or $t\rightarrow-\infty$). There exists a constant $C$ such that $\tilde{I}\left(\gamma(t)\right)=\tilde{I}\left(u(t),v(t)\right)=C$. Since $\tilde{I}$ is a continuous function on $\mathbb{R}^2$, $C=\lim_{t\rightarrow+\infty}\tilde{I}\left(\gamma(t)\right)=\lim_{t\rightarrow+\infty}\tilde{I}\left(u(t),v(t)\right)=0$. This means that $\gamma(t)=\left(u(t),v(t)\right)=(0,0)$. There are no orbits tending or leaving the origin. Thus, the origin of $b\left(\mathcal{X}\right)$ is a monodromic singular point. By statement $(c)$ of Theorem \ref{th-3}, $F$ is a global injective.

\emph {Necessity.} If $F$ is a global injective, then the origin of $b\left(\mathcal{X}\right)$ is a center (also a global center).

From the equation \eqref{eq-38}, we know that $\bar{I}(x,y)\rightarrow 0$ as $(x,y)\rightarrow(0,0)$ along the straight line $l$, that is, $\lim_{r\rightarrow 0}\bar{I}\left(r\cos\theta_0,r\sin\theta_0\right)=0$. For every $\varepsilon>0$, there exists $\delta_1>0$ such that if $0<r<\delta_1$, then
\begin{align}\label{eq-40}
&\left|\bar{I}\left(r\cos\theta_0,r\sin\theta_0\right)\right|<\varepsilon.
\end{align}

 Applying Lemma \ref{le-2}, there exists a $\delta_2>0$ such that the line segment $\overline{OA}=\{\left(r\cos\theta_0,r\sin\theta_0\right):r\in\left[0,\delta_2\right]\}\subset l$ is a transverse section of $b\left(\mathcal{X}\right)$. Let $\delta_3=\min\left\{\delta_1/2,\delta_2/2\right\}$ and $\psi\left(t,\bar{p}\right)$ be the trajectory of vector field $b\left(\mathcal{X}\right)$   passing through the point $\bar{p}=\left(\delta_3\cos\theta_0,\delta_3\sin\theta_0\right)$ for $t\in\mathbb{R}$. Then $\psi\left(t,\bar{p}\right)$ is a periodic orbit. Let $\mathcal{R}$ be a closed region bounded by the $\psi\left(t,\bar{p}\right)$. There exists a neighborhood
$$\mathring{U}\left(O\right)\triangleq\left\{\left(u,v\right):0<u^2+v^2<\delta\right\}$$
such that $\mathring{U}\left(O\right)\subsetneqq\mathcal{R}$.

By definition of the first integral, we have $\bar{I}\left(u_0,v_0\right)=\bar{I}\left(\psi\left(t,\bar{p}_0\right)\right)$ for each $\bar{p}_0=\left(u_0,v_0\right)\in\mathring{U}$ and $t\in\mathbb{R}$. Since $\psi\left(t,\bar{p}_0\right)$ is a periodic orbit and the line segment $\overline{OA}$ is a transverse section, then there exists time $\bar{t}$ such that $\psi\left(\bar{t},\bar{p}_0\right)$ intersects $\overline{OA}$ at the point $\hat{p}_0=\left(\hat{r}\cos\theta_0,\hat{r}\sin\theta_0\right)$ with $0<\hat{r}<\delta_3$, see Figure \ref{fig2}. Using equation \eqref{eq-40}, we get
$$\left|\bar{I}\left(u_0,v_0\right)=\bar{I}\left(\psi\left(t,\bar{p}_0\right)\right)=\bar{I}\left(\hat{r}\cos\theta_0,\hat{r}\sin\theta_0\right)\right|<\varepsilon,$$
that is, $\left|\bar{I}\left(u_0,v_0\right)\right|<\varepsilon$ for arbitrary $\left(u_0,v_0\right)\in\mathring{U}$.

Based on the analysis above, for every $\varepsilon>0$, there exists a neighborhood $\mathring{U}$ such that if $\left(u,v\right)\in\mathring{U}$, then $\left|\bar{I}\left(u,v\right)\right|<\varepsilon$. Therefore, the $\lim_{(u,v)\rightarrow(0,0)}\bar{I}(u,v)$ exists.

This completes the proof of Proposition \ref{pr-6}.
\end{proof}
\begin{figure}[H]
\centering
\begin{minipage}{0.3\linewidth}
\centering
\psfrag{0}{$O$}
\psfrag{1}{$\psi\left(t,\bar{p}\right)$}
\psfrag{2}{$\psi\left(t,\bar{p}_0\right)$}
\psfrag{3}{$\mathcal{R}$}
\psfrag{4}{$\mathring{U}$}
\centerline{\includegraphics[width=1\textwidth]{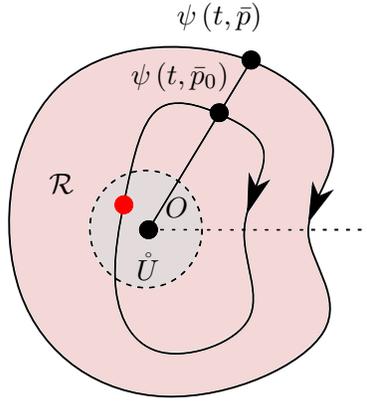}}
\end{minipage}
\caption{The closed region $\mathcal{R}$ and neighborhood $\mathring{U}$.}\label{fig2}
\end{figure}

\begin{proof}[\bf Proof of Theorem \ref{th-14}]
The Theorem \ref{th-14} follows by Proposition \ref{pr-6}.
\end{proof}

\section{Proofs of Theorems \ref{th-16} and \ref{th-13}}\label{se-4}
Our main aim of this section is to prove Theorems \ref{th-16} and to give an alternate proof of Theorem \ref{th-13}. Central to the proof of the Theorem \ref{th-16} is the following lemma, see  Exercise 4 of page 171 of \cite{MR787404} or Theorem 2.2 of \cite{MR3319979}

\begin{lemma}\label{le-4}
 Let $F:\mathbb{R}^n\rightarrow\mathbb{R}^n$ be a $C^1$-map and $\emph{det}DF\left(\mathbf{x}\right)\neq0$ on $\mathbb{R}^n$. Then $F$ is a homeomorphism onto $\mathbb{R}^n$  if and only if $\parallel F\left(\mathbf{x}\right)\parallel\rightarrow\infty$ as $\parallel\mathbf{x}\parallel\rightarrow\infty$.
\end{lemma}

\begin{proof}[\bf Proof of Theorem \ref{th-16}]
  The Theorem \ref{th-16} holds immediately by Lemma \ref{le-4}.
\end{proof}

The rest of this section is to present an alternate proof of Theorem \ref{th-13}.
\begin{lemma}\label{le-6}
  Let $\mathbf{s}=\left(s_1,\ldots,s_n\right)\in\mathbb{N}_+^n$ and $\sum_{j=1}^ny_j^2=1$. If
  \begin{align*}
&x_1=\frac{y_1}{r^{s_1}},\ldots,x_n=\frac{y_n}{r^{s_n}},
  \end{align*}
  then $\sum_{j=1}^nx_j^2\rightarrow\infty$ if and only if $r\rightarrow0$.
\end{lemma}
\begin{proof}
  Without loss of generality, we can assume $s_1\geq\cdots\geq s_n$. For $1>r>0$, one can obtain that
  \begin{align*}
  &\frac{1}{r^{2s_1}}\geq\sum_{j=1}^nx_j^2\geq\frac{1}{r^{2s_n}}.
  \end{align*}
  From this inequation, the lemma follows.
\end{proof}

\begin{proof}[\bf Proof of Theorem \ref{th-13}]
 By Theorem \ref{th-16}, it is enough to prove that
$$\lim_{\parallel \mathbf{x} \parallel\rightarrow \infty}\parallel F\left(\mathbf{x}\right)\parallel=\infty.$$
For convenience, the norm $\parallel\cdot\parallel$ here can take square norm due to the fact that the norms in finite dimensional space are all equivalent. Let $f^i\left(\mathbf{x}\right)=\sum_{l=1}^{m_i}f_l^i\left(\mathbf{x}\right)$ with $i=1,\ldots,n$,  where $f_l^i\left(\mathbf{x}\right)$ is quasi-homogeneous of weighted degree $l$ with respect to weight exponents $\mathbf{s}=\left(s_1,\ldots,s_n\right)$. Here, $F_{\mathbf{s}}\left(\mathbf{x}\right)=\left(f_{m_1}^1,\ldots,f_{m_n}^n\right)$.

Without loss of generality, we can assume $m_1\geq\cdots\geq m_n$. Consider the change of coordinates
$$x_1=\frac{y_1}{r^{s_1}},\ldots,x_n=\frac{y_n}{r^{s_n}}$$
with $r>0$ and $\mathbf{y}=\left(y_1,\ldots,y_n\right)\in\mathbb{S}^{n-1}=\left\{\mathbf{y}:\parallel\mathbf{y}\parallel=1\right\}$.

By Lemma \ref{le-6}, we have
\begin{align}\label{eq-12}
&\lim_{\parallel\mathbf{x}\parallel\rightarrow\infty}\frac{1}{\parallel F\left(\mathbf{x}\right)\parallel}
=\lim_{r\rightarrow0}\frac{1}{\sum\limits_{i=1}^n\left(\sum\limits_{l=1}^{m_i}\frac{1}{r^l}f_l^i\left(\mathbf{y}\right)\right)^2}
=\lim_{r\rightarrow0}\frac{r^{2m_1}}{\sum\limits_{i=1}^n\left(\sum\limits_{l=1}^{m_i}r^{m_1-l}f_l^i\left(\mathbf{y}\right)\right)^2}
\end{align}
with $\mathbf{y}\in\mathbb{S}^{n-1}$. Note that
\begin{align}\label{eq-13}
&\left(\sum\limits_{l=1}^{m_i}r^{m_1-l}f_l^i\left(\mathbf{y}\right)\right)^2
=r^{2m_1-2m_i}\left(\sum\limits_{l=1}^{m_i}r^{m_i-l}f_l^i\left(\mathbf{y}\right)\right)^2.
\end{align}
Since
\begin{align*}
&\lim_{r\rightarrow0}\left(\sum\limits_{l=1}^{m_i}r^{m_i-l}f_l^i\left(\mathbf{y}\right)\right)^2=\left(f_{m_i}^i\left(\mathbf{y}\right)\right)^2,
\end{align*}
there exists a $\delta_i>0$ such that
\begin{align}\label{eq-14}
&\frac{1}{2}\left(f_{m_i}^i\left(\mathbf{y}\right)\right)^2<\left(\sum\limits_{l=1}^{m_i}r^{m_i-l}f_l^i\left(\mathbf{y}\right)\right)^2
\end{align}
for all $0<r<\delta_i$ and $\mathbf{y}\in\mathbb{S}^{n-1}$.

Let $\delta=\min\left\{1/2,\delta_1,\ldots,\delta_n\right\}<1$. Then
\begin{align*}
&0<\frac{r^{2m_1}}{\sum\limits_{i=1}^n\left(\sum\limits_{l=1}^{m_i}r^{m_1-l}f_l^i\left(\mathbf{y}\right)\right)^2}<
\frac{2r^{2m_1}}{\sum\limits_{i=1}^nr^{2m_1-2m_i}\left(f_{m_i}^i\left(\mathbf{y}\right)\right)^2}\leq
\frac{2r^{2m_n}}{\sum\limits_{i=1}^n\left(f_{m_i}^i\left(\mathbf{y}\right)\right)^2}
=\frac{2r^{2m_n}}{\parallel F_{\mathbf{s}}\left(\mathbf{y}\right)\parallel}
\end{align*}
for all $0<r<\delta$ and $\mathbf{y}\in\mathbb{S}^{n-1}$. Since $F_{\mathbf{s}}\left(\mathbf{x}\right)=\mathbf{0}$ has only the trivial solution $\mathbf{x}=\mathbf{0}$ and $\parallel F_{\mathbf{s}}\left(\mathbf{y}\right)\parallel$ is continuous on $\mathbf{y}\in\mathbb{S}^{n-1}$, there exists a $L>0$ such that $L\leq \parallel F_{\mathbf{s}}\left(\mathbf{y}\right)\parallel$ for all $\mathbf{y}\in\mathbb{S}^{n-1}$. Therefore,
$$0\leq\lim_{\parallel\mathbf{x}\parallel\rightarrow\infty}\frac{1}{\parallel F\left(\mathbf{x}\right)\parallel}\leq\lim_{r\rightarrow0}\frac{2r^{2m_n}}{\parallel F_{\mathbf{s}}\left(\mathbf{y}\right)\parallel}=0,\quad\text{that is,}\quad\lim_{\parallel \mathbf{x} \parallel\rightarrow \infty}\parallel F\left(\mathbf{x}\right)\parallel=\infty.$$
This ends the proof.
\end{proof}

\section*{Acknowledgments}
The authors would like to thank Professor Changjian Liu for his valuable suggestions and comments (for example, the information on Lemma \ref{le-4} and the proof details of Theorem \ref{th-13}).

This research is supported by the National Natural Science Foundation of China (No.11971495 and No.11801582) and China Scholarship Council (No.
201906380022).

\end{document}